\documentclass[a4paper, reqno]{amsart}

\usepackage{tgadventor}

\usepackage{amsmath,amsthm,amssymb,relsize}
\usepackage[nobysame,alphabetic, initials]{amsrefs}
\usepackage[margin=30mm]{geometry}
\usepackage{mathrsfs,mathtools,makecell}

\usepackage{extarrows}

\usepackage{booktabs, multirow, adjustbox, array}

\newcolumntype{R}[2]{%
    >{\adjustbox{angle=#1,lap=\width-(#2)}\bgroup}%
    l%
    <{\egroup}%
}
\newcommand{\ra}[1]{\renewcommand{\arraystretch}{#1}}
\usepackage{caption}

\newcommand*\tilt{\multicolumn{1}{R{30}{1em}}}
\usepackage{pifont}
\newcommand{\cmark}{\ding{51}}%
\newcommand{\xmark}{\ding{55}}%
\usepackage{floatrow}

\DeclareFloatFont{tiny}{\small}
\usepackage{enumerate}
\usepackage{tikz}
\usetikzlibrary{
  cd,
  calc,
  positioning,
  fit,
  arrows,
  decorations.pathreplacing,
  decorations.markings,
  shapes.geometric,
  backgrounds,
  bending
}
\usepackage{tikzsymbols}
\PassOptionsToPackage{dvipsnames}{xcolor}
\definecolor{darkblue}{rgb}{0,0,0.6}

\usepackage[breaklinks, pdftex, ocgcolorlinks,colorlinks=true, citecolor=darkblue, filecolor=darkblue, linkcolor=darkblue, urlcolor=darkblue]{hyperref}
\usepackage[capitalize,noabbrev]{cleveref}

\makeatletter
\newcommand{\mylabel}[2]{#2\def\@currentlabel{#2}\label{#1}}
\makeatother

\makeatletter
\newtheorem*{rep@theorem}{\rep@title}
\newcommand{\newreptheorem}[2]{%
\newenvironment{rep#1}[1]{%
 \def\rep@title{#2 \ref{##1}}%
 \begin{rep@theorem}}%
 {\end{rep@theorem}}}
\makeatother

\newtheorem{proposition}{Proposition}[section]
\newtheorem{theorem}[proposition]{Theorem}
\newtheorem{corollary}[proposition]{Corollary}
\newtheorem{lemma}[proposition]{Lemma}

\theoremstyle{definition}

\newtheorem{question}[proposition]{Question}

\theoremstyle{remark}
\newtheorem{remark}[proposition]{Remark}
\newtheorem{claim}[proposition]{Claim}
\newtheorem*{remark*}{Remark}

\newreptheorem{theorem}{Theorem}
\newreptheorem{lemma}{Lemma}
\newreptheorem{proposition}{Proposition}
\newreptheorem{corollary}{Corollary}
\newreptheorem{question}{Question}
\numberwithin{equation}{section}

\newcommand{\Diff}{\mathrm{Diff}}
\newcommand{\Top}{\mathrm{Top}}

\newcommand{\CC}{\mathbb{C}}

\newcommand{\R}{\mathbb{R}}
\newcommand{\Z}{\mathbb{Z}}

\newcommand{\cN}{\mathcal{N}}
\newcommand{\LL}{\mathbb{L}}

\newcommand{\im}{\operatorname{Im}}
\newcommand{\Id}{\operatorname{Id}}

\newcommand{\into}{\hookrightarrow}
\newcommand{\ol}{\overline}
\newcommand{\wt}{\widetilde}

\newcommand{\sm}{\setminus}

\newcommand{\ks}{\operatorname{ks}}

\newcommand{\CP}{\mathbb{CP}}
\newcommand{\RP}{\mathbb{RP}}
\renewcommand{\star}{\mathop{}\!*}
\DeclareMathOperator{\capp}{cap}

\newcommand{\tmfrac}[2]{\mbox{\large$\frac{#1}{#2}$}} 

\usepackage{letltxmacro}

\LetLtxMacro\Oldfootnote\footnote

\begin{document}
\title{Gluck twists on concordant or homotopic spheres}

\author{Daniel Kasprowski}
\address{School of Mathematical Sciences, University of Southampton, United Kingdom}
\email{d.kasprowski@soton.ac.uk}

\author{Mark Powell}
\address{School of  Mathematics and Statistics, University of Glasgow, United Kingdom}
\email{mark.powell@glasgow.ac.uk}

\author{Arunima Ray}
\address{Max Planck Institut f\"{u}r Mathematik, Vivatsgasse 7, 53111 Bonn, Germany}
\email{aruray@mpim-bonn.mpg.de }

\def\subjclassname{\textup{2020} Mathematics Subject Classification}
\expandafter\let\csname subjclassname@1991\endcsname=\subjclassname
\subjclass{
57K40, 
57N70, 
57R80. 
}
\keywords{Gluck twist, 4-manifolds}

\begin{abstract}
Let $M$ be a compact 4-manifold and let $S$ and $T$ be embedded $2$-spheres in $M$, both with trivial normal bundle.  We write $M_{S}$ and $M_T$ for the 4-manifolds obtained by the Gluck twist operation on $M$ along $S$ and $T$ respectively.  We show that if $S$ and  $T$ are concordant, then $M_S$ and $M_T$ are $s$-cobordant, and so if $\pi_1(M)$ is good, then $M_S$ and $M_T$ are homeomorphic. Similarly, if $S$ and $T$ are homotopic then we show that $M_S$ and $M_T$ are simple homotopy equivalent.
Under some further assumptions, we deduce that~$M_S$ and $M_T$ are homeomorphic. We show that additional assumptions are necessary by giving an example where $S$ and $T$ are homotopic but $M_S$ and $M_T$ are not homeomorphic.  We also give an example where $S$ and $T$ are homotopic and $M_S$ and $M_T$ are homeomorphic but not diffeomorphic.
\end{abstract}
\maketitle

\section{Introduction}

Given a compact, smooth $4$-manifold $M$ and a smooth embedding of a $2$-sphere in the interior of $M$ with image $S$, such that the submanifold $S$ has trivial normal bundle, we may obtain a 4-manifold $M_S$ as follows.  Fix an identification of a closed tubular neighbourhood of $S$ as $\ol{\nu} S \cong D^2 \times S^2$ and let $\partial_S(M\sm \nu S)\subseteq \partial (M\sm \nu S)$ denote the boundary component identified with $\partial \ol\nu{S}\cong S^1 \times S^2$. Let $R \colon S^1 \to SO(3)$ be a smooth map that represents a generator of $\pi_1(SO(3))$.  The \emph{Gluck twist} on $S^1 \times S^2$, named after Herman Gluck~\cite{gluck62}, is the diffeomorphism
\begin{align*}
G \colon S^1 \times S^2 &\to S^1 \times S^2\\
(x,y) &\mapsto (x,R(x) \cdot y),
\end{align*}
obtained by restricting the canonical action of $SO(3)$ to $S^2 \subseteq \R^3$.
By definition the \emph{Gluck twist operation} on $M$ along $S$ consists of replacing $M$ by the smooth 4-manifold
\[M_S := (M \sm \nu S) \cup_{G} (D^2 \times S^2).\]
Here we identify using  $\partial (D^2 \times S^2) = S^1 \times S^2 \xrightarrow{G} S^1 \times S^2 \cong \partial_S(M \sm \nu S)$.
The Gluck twist on a null-homologous 2-sphere in a simply connected 4-manifold does not change the homotopy type, but this is not the case in general.
Nevertheless, $\pi_1(M_S) \cong \pi_1(M)$, as we show in~\cref{lem:pi1}. There is also a canonical identification $\partial M_S=\partial M$.
The construction can be carried through in exactly the same way in the topological category, provided~$S$ is a topological submanifold of $M$, i.e.\ the image of a locally flat embedding with trivial normal bundle. Note that we allow nonorientable 4-manifolds, as well as those with nonempty boundary.

A classical question compares the diffeomorphism type of $M$ and $M_S$ for a $2$-knot $S\subseteq M$, most commonly in the case of $M=S^4$.
Previous work on Gluck twists in 4-manifolds other than $S^4$ includes \cite{akbulut-fake-gluck,akbulut-yasui}.

We consider the effect of the Gluck twisting on pairs $\{S,T\}$ of $2$-knots in some $M$, and compare the homeomorphism type and homotopy type of~$M_S$ and~$M_T$. One hope is that by generalising the question, we might inspire readers to find new exotic 4-manifolds using Gluck twists.
First we consider the case of concordant $2$-spheres $S$ and $T$ in a $4$-manifold $M$, i.e.~when there is a proper, locally flat embedding $c \colon S^2 \times [0,1] \to M \times [0,1]$ with $c(S^2 \times \{0\})= S \subseteq M \times \{0\}$ and
$c(S^2 \times \{1\}) = T \subseteq M \times \{1\}$.

\begin{theorem}\label{thm:conc-gluck-thm}
The following statement holds in both the smooth category and the topological category with locally flat embeddings.
 Let $M$ be a compact $4$-manifold, and let $S,T\subseteq M$ be embedded, concordant $2$-spheres, both with trivial normal bundle.
Then $M_S$ and $M_T$ are $s$-cobordant.

If~$\pi_1(M)$ is a good group, it follows that $M_S$ and $M_T$ are homeomorphic.
\end{theorem}

We will prove \cref{thm:conc-gluck-thm} in \cref{sec:conc-gluck-thm}.
For the deduction that $M_S$ and $M_T$ are homeomorphic in~\cref{thm:conc-gluck-thm}, we use the $s$-cobordism theorem of Freedman and Quinn~\cite{FQ}*{Theorem~7.1A}, which states that every $5$-dimensional compact $s$-cobordism with good fundamental group is homeomorphic to a product. A group is said to be \emph{good}, in the sense of Freedman, if it satisfies the $\pi_1$-null disc property~\cite{Freedman-Teichner:1995-1}. At present, the class of good groups is known to contain groups of subexponential growth~\cite{Freedman-Teichner:1995-2,Krushkal-Quinn:2000-1} and to be closed under subgroups, quotients, extensions, and colimits.  As concrete examples, the class includes finite groups and elementary amenable groups, and therefore in particular all solvable groups. It is not known to include nonabelian free groups. See~\cite{DET-book-goodgroups} for a comprehensive survey on good groups.

For our second result, we consider the case of homotopic $2$-knots in a given manifold.
Note that concordant $2$-knots are homotopic.

\begin{theorem}\label{thm:gluck-she}
Let $M$ be a compact $4$-manifold, and let $S,T\subseteq M$ be locally flat, embedded, and homotopic $2$-spheres, both with trivial normal bundle.
Then $M_S$ and $M_T$ are simple homotopy equivalent via a simple homotopy equivalence that restricts to the identity on the canonically identified boundaries $\partial M_S = \partial M_T$.
\end{theorem}

We will prove \cref{thm:gluck-she} in \cref{sec:gluck-she}.
For certain fundamental groups, simple homotopy equivalence can be upgraded to homeomorphism. This includes some fundamental groups for which there is an explicit homeomorphism classification~\cites{F,FQ,Hambleton-Kreck:1988-1,Wang-nonorientable,HKT}. However, the class of known good groups is much larger than the class of groups for which explicit classification results are available, and for many good groups, even in the absence of a known homeomorphism classification, it is possible to use the surgery exact sequence to conclude that simple homotopy equivalent $4$-manifolds are homeomorphic. We describe the outcome of both methods in the following corollary.

Recall that an identification of the fundamental group $\pi_1(M)$ of a $4$-manifold $M$ with a group~$\pi$ determines a map $c_M \colon M \to B\pi$, up to homotopy, classifying the universal cover, where $B\pi \simeq K(\pi,1)$ is the classifying space.

\begin{corollary}\label{cor:gluck-homeo}
Let $M$ be a compact $4$-manifold, and let $S,T\subseteq M$ be locally flat, embedded, and homotopic $2$-spheres, both with trivial normal bundle. Let $\pi:= \pi_1(M)$ and let $w \colon \pi \to C_2 = \{\pm 1\}$ be the orientation character. Suppose in addition that one of the following holds.
\begin{enumerate}[(i)]
  \item\label{item:i} $M$ is closed, orientable, and $\pi$ is cyclic.
  \item\label{item:iii} The following four conditions hold.
  \begin{enumerate}
      \item[\mylabel{item:G}{(G)}] $\pi$ is good;
      \item[\mylabel{item:A4}{(A4)}] the assembly map $A_4 \colon H_4(\pi;\LL\langle1\rangle^w)\to L_4^s(\Z\pi,w)$ is injective;
      \item[\mylabel{item:A5}{(A5)}] the assembly map $A_5\colon H_5(\pi;\LL\langle1\rangle^w)\to L_5^s(\Z\pi,w)$ is surjective;
      \item[\mylabel{item:H}{(H)}] the map $(c_M)_* \colon H_3(M;\Z/2)\to H_3(\pi;\Z/2)$ is surjective, for some choice of classifying map $c_M$.
  \end{enumerate}
\end{enumerate}
Then $M_S$ and $M_T$ are homeomorphic.
\end{corollary}

The conditions of \cref{cor:gluck-homeo} are in fact conditions under which simple homotopy equivalent $4$-manifolds with equal Kirby--Siebenmann invariants are homeomorphic, irrespective of whether they arise via Gluck twisting. We explain the conditions in \cref{cor:gluck-homeo}\,\eqref{item:iii} further in \cref{rem:conditions}. For now we limit ourselves to pointing out that Condition ~\eqref{item:iii} is satisfied for the solvable Baumslag--Solitar groups $\langle a,b \mid aba^{-1} = b^k \rangle \cong  \Z\big[\tmfrac{1}{k}\big] \rtimes \Z$, including $\Z\oplus \Z$ when $k=1$. Condition~\ref{item:H} is not satisfied for $\Z\oplus \Z\oplus \Z$.  Details on these examples, and more, can also be found in \cref{rem:conditions}.

The assumptions in~\cref{thm:conc-gluck-thm,thm:gluck-she}, that $S$ and $T$ are concordant/homotopic, are in general necessary, since Gluck twisting along spheres that are not homotopic can result in homotopy inequivalent 4-manifolds.
For example, Gluck twisting along $S^2\times \{p\}\subseteq S^2\times S^2$ results in the twisted fibre bundle $S^2 \wt{\times} S^2$, whereas Gluck twisting along an unknotted $S^2$ returns $S^2 \times S^2$ again (see \cref{prop:unknotted-gluck} below). These 4-manifolds are not homotopy equivalent and so are certainly not $s$-cobordant. For non-homotopic 2-spheres, Gluck twisting can result in homotopy equivalent manifolds with different homeomorphism types.
Fan~\cite{fan-thesis}*{Theorem~4.5} gave an example of a simply connected $4$-manifold with nonempty boundary, together with a homotopically essential embedded 2-sphere, such that the Gluck twist preserves the homotopy type rel.\ boundary but changes the homeomorphism type. See \cref{glucktable} for a summary of these, and other, examples.

As a counterpoint to \cref{cor:gluck-homeo}, we give new examples in this direction, with the additional properties that the spheres are homotopic and the ambient manifold is closed, as follows.

\begin{proposition}\label{thm:gluck-giving-fake}
There exists a closed, orientable $4$-manifold $M$ and locally flat, embedded, and homotopic $2$-spheres $S,T\subseteq M$, both with trivial normal bundle, such that $M_S$ and $M_T$ are simple homotopy equivalent but not homeomorphic.
\end{proposition}

\begin{table}
\ra{1.5}
    \begin{tabular}{@{}p{45mm}p{10mm}p{10mm}p{10mm}p{10mm}p{10mm}p{17mm}}
        \toprule
            &\tilt{$S$, $T$ homotopic}&\tilt{$M$ orientable}&\tilt{$M$ closed} &\tilt{$M_S\simeq_s M_T$}    &\tilt{$M_S\cong_\Top M_T$}&\tilt{$M_S\cong_\Diff M_T$}\\
            \midrule
        \makecell{$M=S^2\times S^2$, $S=S^2\times \{*\}$,\\ $T$=unknot.}  &\xmark &\cmark &\cmark &\xmark &\xmark &\xmark\\
        \makecell{$M=(S^1\wt\times S^3) \# (S^2\times S^2)$,\\ $T$=unknot~\cite{akbulut-fake-gluck} cf.~~\cite{torres-gluck}.}  &\xmark &\xmark &\cmark &\cmark &\cmark &\xmark\\
        Examples by Fan~\cite{fan-thesis}.  &\xmark &\cmark &\xmark &\cmark &\xmark &\xmark\\
        Examples from Prop.~\ref{thm:gluck-giving-fake}.  &\cmark &\cmark &\cmark &\cmark &\xmark &\xmark\\
        Examples from Prop.~\ref{thm:gluck-gives-exotic}.  &\cmark &\xmark &\cmark &\cmark &\cmark &\xmark\\
        \bottomrule
    \end{tabular}
\caption{The properties of some interesting examples of Gluck twists. The symbol for the equivalence relations in the top row are defined in \cref{section:conventions}.}
\label{glucktable}
\end{table}

\subsection{Effect on diffeomorphism type}

The following proposition is well-known, but important to observe. It shows that the classical idea of comparing $M_S$ and $M$ is a special case of comparing $M_S$ and $M_T$, by taking $T$ to be an unknotted 2-sphere.

\begin{proposition}\label{prop:unknotted-gluck}
Let $M$ be a compact, smooth $4$-manifold. If $T \subseteq M$ is an unknotted smooth $2$-knot, then $M$ and $M_T$ are diffeomorphic.
\end{proposition}

\begin{proof}
First let $M= S^4$. Then $M_T \cong M$ because the Gluck twist $G$ extends to a diffeomorphism of $S^1 \times D^3$, the complement in $S^4$ of the trivial $2$-knot $S$.  Since every orientation-preserving diffeomorphism can be isotoped to fix a ball, this implies that for an unknotted $2$-sphere $T$ we have a rel.\ boundary diffeomorphism $D^4_T \cong D^4$. Now for a general $M$, every trivial $2$-knot $T$ is contained in a $4$-ball, and the result follows.
\end{proof}

Gluck twisting along nontrivial, smooth $2$-knots in $S^4$ is a well-known method to produce smooth homotopy 4-spheres, which could potentially be counterexamples to the smooth $4$-dimensional Poincar\'{e} conjecture. Progress thus far has been in the opposite direction: there are several families of nontrivial 2-knots in $S^4$ for which it has been proven that Gluck twisting along them yields $S^4$ again~\cites{gluck62,gordon76,pao78,litherland79,nashstip,naylorschwartz}.

The topological $4$-dimensional Poincar\'{e} conjecture or the homeomorphism classification of simply connected $4$-manifolds is usually cited to prove that Gluck twisting on $S^4$ preserves homeomorphism type. Since every $2$-knot in $S^4$ is concordant to the trivial $2$-knot~\cite{kervaireslice}, the combination of \cref{thm:conc-gluck-thm} and \cref{prop:unknotted-gluck} gives a different route. This method, and also that of \cref{cor:gluck-homeo}, both have the advantage that they can be applied to the study of Gluck twisting on 2-knots in more general non-simply connected $4$-manifolds.  We hope that readers will be motivated to investigate this.

For non-homotopic 2-spheres, Gluck twisting can change the smooth structure of a fixed 4-manifold. Akbulut~\cite{akbulut-fake-gluck} constructed an example of a homotopically essential $2$-knot $S$ in $P:=(S^1 \wt{\times} S^3) \# (S^2 \times S^2)$ such that Gluck twisting changes the smooth structure. That is $Q:= P_S$ is homeomorphic~\cite{Wang-nonorientable} but not diffeomorphic to $P$.  Torres~\cite{torres-gluck} extended this construction to other nonorientable $4$-manifolds.

The following proposition sharpens this result by exhibiting two homotopic spheres $S$ and $T$ in a closed, nonorientable $4$-manifold~$M$ such that Gluck twisting gives rise to an exotic pair $M_S$ and $M_T$.  Our construction makes use of the Cappell--Shaneson exotic $\RP^4$~\cites{CS-new-four-mflds-bams,CS-new-four-mflds-annals}, which we denote by $R$. Recall that two $4$-manifolds $M$ and $N$ are said to be \emph{stably diffeomorphic} if there exist $n,n'$ such that $M\# n(S^2\times S^2)\cong N\# n'(S^2\times S^2)$.

\begin{proposition}\label{thm:gluck-gives-exotic}
There exists a pair of homotopic smooth $2$-knots $S$ and $T$ in $M := \RP^4 \# S^2 \wt{\times} S^2$, both with trivial normal bundle, such that $M_S \cong \RP^4 \# (S^2 \times S^2)$ and $M_T \cong R \# (S^2 \times S^2)$. In particular $M_S$ and $M_T$ are homeomorphic but not stably diffeomorphic.
\end{proposition}

The proof of this proposition uses neither \cref{thm:conc-gluck-thm} nor \cref{cor:gluck-homeo}.  We also remark that we do not produce a new example of an exotic pair of $4$-manifolds, but rather the proof uses the fact that $\RP^4$ and $R$ are known to not be stably diffeomorphic.

At present, there is no known example of the Gluck twisting operation changing the smooth structure of an orientable $4$-manifold.

\begin{question}\label{question-on-gluck}
Does there exist a smooth, \emph{orientable} $4$-manifold $M$ and a pair of smooth $2$-knots $S$ and $T$ in $M$ with trivial normal bundle such that $M_S$ and $M_T$ are homeomorphic but not diffeomorphic?
\end{question}

Akbulut and Yasui give a simple criterion when a Gluck twist on a smooth $4$-manifold with odd intersection form does not change its diffeomorphism type~\cite{akbulut-yasui}.

\cref{thm:conc-gluck-thm,cor:gluck-homeo}  provide methods to potentially show that $M_S$ and $M_T$ are homeomorphic, in the case that $S$ and $T$ are either homotopic or concordant. We hope that they will be useful to help answer Question~\ref{question-on-gluck}.

\subsection{Conventions}\label{section:conventions}

All $4$-manifolds are assumed to be connected and based, with a local orientation at the basepoint. Given a pair of manifolds, we use the symbol $\simeq$ for homotopy equivalence, $\simeq_s$ for simple homotopy equivalence, $\cong_\Top$ for homeomorphism, and $\cong_\Diff$ for diffeomorphism. When clear from context, we use $\cong$ instead of $\cong_\Top$ and $\cong_\Diff$. We use $C_2$ to denote the multiplicative group $\{\pm 1\}$.

\subsection*{Outline}
\cref{thm:conc-gluck-thm} is proven in \cref{sec:conc-gluck-thm}, and \cref{thm:gluck-she} in \cref{sec:gluck-she}.  In~\cref{sec:gluck-homeo} we prove~\cref{cor:gluck-homeo}. \cref{thm:gluck-giving-fake,thm:gluck-gives-exotic} are proven in~\cref{sec:fake-exotic}.

\subsection*{Acknowledgements}

Thanks to Selman Akbulut for helping to clarify the state of the art with regards to Gluck twisting.
We thank the Max Planck Institute for Mathematics in Bonn and Durham University for hospitality.
Finally we are grateful to an anonymous referee for their comments which helped us to improve the exposition.

DK was supported by the Deutsche Forschungsgemeinschaft under Germany's Excellence Strategy - GZ 2047/1, Projekt-ID 390685813.
MP was partially supported by EPSRC New Investigator grant EP/T028335/1 and EPSRC New Horizons grant EP/V04821X/1.

\section{\texorpdfstring{$s$-cobordism}{s-cobordism}}\label{sec:conc-gluck-thm}

In this section we prove \cref{thm:conc-gluck-thm}, that for $S$ and $T$ concordant we have an $s$-cobordism between $M_S$ and $M_T$.
We start with a straightforward lemma.

\begin{lemma}\label{lem:pi1}
Let $M$ be a compact $4$-manifold, and let $S\subseteq M$ be a locally flat embedded $2$-sphere with trivial normal bundle. Then $\pi_1(M) \cong \pi_1(M_S)$.
\end{lemma}

\begin{proof}
The operation of attaching $D^2\times S^2$ to a 4-manifold along $S^1\times S^2$ can be decomposed into a $2$-handle attachment followed by a $4$-handle attachment. Both $M$ and $M_S$ are obtained from $M\sm \nu S$ by gluing in a copy of $D^2\times S^2$, and in both cases the $2$-handle is attached along a meridian of $S$ in $\partial\ol\nu S$. In the case of $M_S$, we can use the meridian at a fixed point of the rotation $R \colon S^1 \to SO(3)$.
Then both $\pi_1(M)$ and $\pi_1(M_S)$ are isomorphic to the quotient of $\pi_1(M \sm \nu S)$ by the normal subgroup generated by a meridian of $S$, and so in particular are isomorphic to each other.
\end{proof}

Having dealt with the fundamental group, we give the proof of \cref{thm:conc-gluck-thm}.
The proof in the smooth category uses handle decompositions and we do this case first. The relevant tools, such as topological transversality, the existence of normal bundles for locally flat, embedded submanifolds of $4$-manifolds, and the existence of topological handle decompositions for $5$-manifolds, are also available in the topological category, due to Quinn~\cite{Quinn-annulus} (see also ~\citelist{\cite{Kirby-Siebenmann:1977-1}\cite{FQ}\cite{freedman-book-flowchart}}).  The end of the upcoming proof explains how to use these tools to reduce the proof in the topological category to that in the smooth category.

\begin{proof}[Proof of \cref{thm:conc-gluck-thm}]
Suppose that $S$ and $T$ are concordant $2$-knots in $M$, and let $C$ be a concordance between them. Recall this means that $C$ is the image of a proper embedding $c \colon S^2 \times [0,1] \into M \times [0,1]$ with $c(S^2 \times [0,1]) \cap (M \times \{0\}) = c(S^2 \times \{0\}) = S$ and
$c(S^2 \times [0,1]) \cap (M \times \{1\}) = c(S^2 \times \{1\}) = T$. We first assume that $M$ and  $C$ are smooth.

For a submanifold $Y\subseteq Z$ with an open tubular neighbourhood $\nu Y$, denote the exterior of $Y$ by $E_Y$.   So in particular $E_S := M \sm \nu S$, $E_T := M \sm \nu T$, and $E_C := (M \times [0,1]) \sm \nu C$.
Identify a closed tubular neighbourhood $\ol{\nu} C$ of $C$ with $S^2 \times D^2 \times [0,1]$, and define $W:=(M \times [0,1])_C$, the result of gluing $S^2 \times D^2 \times [0,1]$ to $E_C$ via
\[G \times \Id_{[0,1]} \colon S^1 \times S^2 \times [0,1] \to \partial \ol{\nu} C \subseteq \partial E_C.\]
We will show that $W$ is a rel.\ boundary  $s$-cobordism between $M_S$ and $M_T$. By definition, this means that $W$ is a manifold with corners, with $\partial W \cong M_S \cup (\partial M \times [0,1]) \cup M_T$, and the inclusion maps $M_S\into W$ and $M_T\into W$ are simple homotopy equivalences. By symmetry it suffices to show that $M_S\into W$ is a simple homotopy equivalence; the proof for $M_T \into W$ is identical with $T$ replacing $S$. To prove this, we will show that the inclusion induced map $\pi_1(M_S) \to \pi_1(W)$ is an isomorphism, and that the relative chain complex $C_*(W,M_S;\Z[\pi_1(W)])$ is acyclic and has trivial Whitehead torsion.

First note that by~\cref{lem:pi1}, $\pi_1(M)\cong \pi_1(M_S) =: \pi$.  We assert that $\pi_1(W) \cong \pi$ and that the inclusion $M_S \into W$ induces an isomorphism.
Consider the following commutative diagram, where all the maps are induced by inclusion:
\begin{equation}\label{eq:scob-diag}
\begin{tikzcd}
\pi_1(M_S) \ar[d] & \pi_1(E_S) \ar[d] \ar[l, two heads] \ar[r, two heads] & \pi_1(M) \ar[d,"\cong"] \\
\pi_1(W) & \pi_1(E_C) \ar[l, two heads] \ar[r, two heads] & \pi_1(M \times [0,1]).
\end{tikzcd}
\end{equation}
Since the horizontal maps are surjective, the left and right vertical maps are those induced by the inclusion $E_S \to E_C$.
Let $K_S := \ker (\pi_1(E_S) \to \pi_1(M))$ and $K_C := \ker (\pi_1(E_C) \to \pi_1(M\times [0,1]))$.
Then we have an induced map $\iota \colon \pi_1(E_S)/K_S \to \pi_1(E_C)/K_C$, and since $\iota$ is identified via the inclusion maps with the right vertical map $\pi_1(M) \xrightarrow{\cong} \pi_1(M \times [0,1])$ in~\eqref{eq:scob-diag}, we learn that $\iota$ is an isomorphism.
The kernels of each of the horizontal maps in~\eqref{eq:scob-diag} are the normal subgroups generated by a meridian of $S$ or $C$, as appropriate, and as described in the proof of \cref{lem:pi1}. Thus $\ker(\pi_1(E_S) \to \pi_1(M_S))=K_S$ and $\ker(\pi_1(E_C) \to \pi_1(W))=K_C$.  We deduce from the latter that $\pi_1(M \times [0,1]) \cong \pi_1(W)$, and moreover that $\iota$ can also be identified via the inclusion maps with the left vertical map $\pi_1(M_S) \to \pi_1(W)$. So this map is an isomorphism as asserted.

Next we build a convenient handle decomposition of $M\times [0,1]$.  Consider the concordance exterior $E_C$ as a manifold with corners, where the boundary decomposes as \[\partial E_C = E_S \cup (C \times S^1) \cup (\partial M\times [0,1]) \cup E_T.\]  That is, we consider $E_C$ as a cobordism rel.\ boundary between $E_S$ and $E_T$.  Choose a handle decomposition of $E_C$ relative to $E_S$.  Note that `relative to $E_S$' means that no handles are attached to $E_T \cup (\partial M \times [0,1]) \cup (C \times S^1)$. Glue on the product $C\times D^2$ using the identity map on $C\times S^1$ to obtain a handle decomposition of $M\times [0,1]$ relative to $M$. Similarly glue $C\times D^2$ using $G\times \Id_{[0,1]}$ to obtain a handle decomposition of $W$ relative to $M_S$. In both cases, all the handles are located within the $E_C$ piece and are attached away from $E_T \cup (\partial M \times [0,1]) \cup (C \times S^1)$.

Corresponding to the handle decompositions constructed above we have relative chain complexes $C_*(M\times [0,1],M;\Z\pi)$ and $C_*(W,M_S;\Z\pi)$. The chain complex $C_*(M\times [0,1],M;\Z\pi)$ is acyclic, and has trivial Whitehead torsion, since $M\times [0,1]$ is a product. By construction, in particular because the handles are attached away from $E_T \cup (\partial M \times [0,1]) \cup (C \times S^1)$,  the chain complex $C_*(W,M_S;\Z\pi)$ has the same handle attaching maps as $C_*(M\times [0,1],M;\Z\pi)$ and therefore is also acyclic and has trivial Whitehead torsion. We have now shown that $M_S \into W$ is a simple homotopy equivalence, and therefore that $W$ is a rel.\ boundary smooth $s$-cobordism between $M_S$ and $M_T$.

As mentioned in the introduction, to deduce that for $\pi_1(M)$ a good group $M_S$ and $M_T$ are homeomorphic, we apply the $s$-cobordism theorem~\cite{FQ}*{Theorem~7.1A}, which completes the proof in the smooth category.

Next we extend the proof to the case where $M$ need not be smooth, and $C$ is only locally flat. The only salient difference is in the construction of an appropriate handle decomposition for the claimed $s$-cobordism. The concordance $C$ has a normal bundle $\nu C$ in $M\times [0,1]$ by~\cite{ks-codim2-normal-bundles}, extending normal bundles for $S$ and $T$ in $M\times \{0,1\}$, which exist by work of Quinn~\citelist{\cite{Quinn-annulus}\cite{FQ}*{Chapter~9}}. As before write $E_S := M \sm \nu S$, $E_T := M \sm \nu T$, and $E_C := (M \times [0,1]) \sm \nu C$. Identify  $\ol{\nu} C$ with $S^2 \times D^2 \times [0,1]$ and define $W:= S^2 \times D^2 \times [0,1] \cup_{G \times \Id_{[0,1]}} E_C$.  Consider $E_C$ as a cobordism rel.\ boundary between $E_S$ and $E_T$.  By~\cite{FQ}*{Theorem~9.1}, there is a topological handle decomposition of $E_C$ relative to $E_S$. The rest of the proof is directly analogous to the smooth case.
\end{proof}

\begin{remark}
The proof of \cref{thm:conc-gluck-thm} extends to other settings. For example two logarithmic transforms~\cite{GompfStip}*{Sections~3.3 and~8.3} on concordant embedded tori $T_1, T_2\subseteq M$ with trivial normal bundle, and equal gluing coefficients, give rise to $s$-cobordant $4$-manifolds.  Similarly Price twists~\cite{pricetwist} on concordant embeddings of $\RP^2$ give rise to $s$-cobordant $4$-manifolds. We leave the details of these extensions to the reader.
\end{remark}

\section{Simple homotopy equivalence}\label{sec:gluck-she}

In this section we prove \cref{thm:gluck-she}. The theorem states that for $S$ and $T$ homotopic, the manifolds $M_S$ and $M_T$ are simple homotopy equivalent.  We begin with a lemma, which contains the key observation. In the smooth category, the diffeomorphism of the ambient manifolds $M_S\#\CP^2 \cong M\#\CP^2$ in the statement is well-known (see e.g.~\cite{GompfStip}*{Exercise~5.2.7\,(b)}) but we give a proof for the reader's convenience, to verify the additional claim about the sphere $S\#\CP^1$, and to give the argument in the topological category.

Recall that a homotopy equivalence $f$ between CW complexes is said to be a simple homotopy equivalence if the Whitehead torsion $\tau(f)$ vanishes. While it is open whether every topological $4$-manifold is homeomorphic to a $4$-dimensional CW complex, it is still possible to define the notion of simple homotopy equivalence for topological $4$-manifolds by embedding them in high-dimensional Euclidean space. For more details on this, see~\citelist{\cite{Kirby-Siebenmann:1977-1}*{Essay III, Section 4}\cite{KPR-table}*{Section~2}}.

\begin{lemma}\label{lemma:CP2-stable-diffeo}
Let $M$ be a compact $4$-manifold and let $S\subseteq M$ be a locally flat embedded $2$-sphere with trivial normal bundle.
There is a homeomorphism
of pairs
\[
\Psi_S\colon (M_S,\emptyset) \# (\CP^2,\CP^1)\xlongrightarrow{\cong} (M,S)\# (\CP^2,\CP^1),
\]
that
restricts to the identity on the naturally identified boundaries $\partial M = \partial M_S$.
\end{lemma}

\begin{proof}
Choosing a trivialisation of the normal bundle of $S$ gives rise to an identification of a closed tubular neighbourhood $\ol{\nu} S$ with $S^2 \times D^2$. The standard smooth structure on $S^2 \times D^2$ induces a smooth structure on $\ol{\nu}S$.  Let $E_S := M \sm \nu S$ be the exterior of $S$.  We may and shall assume that the connected sum operation has been performed within $\ol{\nu}S=S^2 \times D^2$, but away from $S$, to yield $M \# \CP^2 \cong E_S \cup_{S^2 \times S^1} ((S^2 \times D^2) \# \CP^2)$.  We will work inside $(S^2 \times D^2) \# \CP^2$, using its smooth structure, and therefore we may apply the methods of handle calculus.  \cref{fig:local-handle-diagram}\,(a) shows a handle diagram for $(S^2 \times D^2) \# \CP^2$ relative to the boundary $S^2 \times S^1$, using the notation of~\cite{GompfStip}*{Section~5.5}. That is, we start with $S^2 \times S^1 \times [0,1]$, attach two 2-handles to $S^2 \times S^1 \times \{1\}$, and then a single 4-handle.

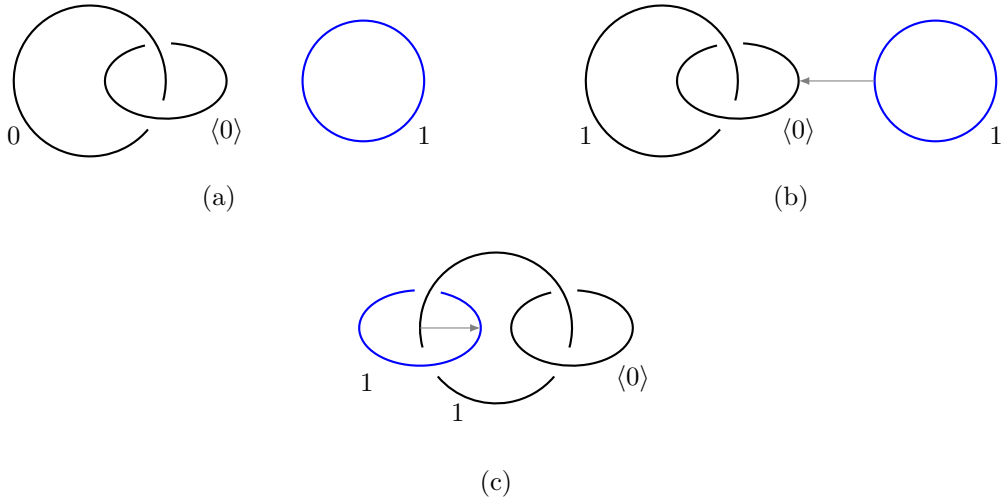
\begin{figure}[htb]
\centering
	\begin{tikzpicture}[thick]%
\tikzset{
    partial ellipse/.style args={#1:#2:#3}{
        insert path={+ (#1:#3) arc (#1:#2:#3)}
    }
}
\draw[] (0,0) [partial ellipse=-250:85:8mm and 5mm];
\draw[] (-1,0) [partial ellipse=-15:320:1cm and 1cm];
\draw[blue] (2.6,0) circle (8mm);
\node[label=above:(a)] (a) at (0.7,-2) {};
\node[label=above:$\langle 0\rangle$] (<0>) at (0.8,-1.1) {};
\node[label=above:0] (sphere) at (-2,-1.1) {};
\node[label=above:1] (cp2) at (3.4,-1.1) {};
\end{tikzpicture}\hspace{1.6cm}
\begin{tikzpicture}[thick]%
\tikzset{
    partial ellipse/.style args={#1:#2:#3}{
        insert path={+ (#1:#3) arc (#1:#2:#3)}
    }
}
\draw[] (0,0) [partial ellipse=-250:85:8mm and 5mm];
\draw[] (-1,0) [partial ellipse=-15:320:1cm and 1cm];
\draw[blue] (2.6,0) circle (8mm);
\draw[gray,-latex, thin] (1.8,0)--(0.8,0);
\node[label=above:(b)] (b) at (0.7,-2) {};
\node[label=above:$\langle 0\rangle$] (<0>) at (0.8,-1.1) {};
\node[label=above:1] (sphere) at (-2,-1.1) {};
\node[label=above:1] (cp2) at (3.4,-1.1) {};
\end{tikzpicture}
\begin{tikzpicture}[thick]%
\tikzset{
    partial ellipse/.style args={#1:#2:#3}{
        insert path={+ (#1:#3) arc (#1:#2:#3)}
    }
}
\draw[] (0,0) [partial ellipse=-250:85:8mm and 5mm];
\draw[] (-1,0) [partial ellipse=220:320:1cm and 1cm];
\draw[] (-1,0) [partial ellipse=-15:195:1cm and 1cm];
\draw[blue] (-2,0) [partial ellipse=-265:70:8mm and 5mm];
\draw[gray,-latex, thin] (-2,0)--(-1.2,0);
\node[label=above:(c)] (c) at (-1,-2.5) {};
\node[label=above:$\langle 0\rangle$] (<0>) at (0.8,-1.1) {};
\node[label=above:1] (sphere) at (-1.5,-1.5) {};
\node[label=above:1] (cp2) at (-2.7,-1.1) {};
\end{tikzpicture}
\caption{(a) A handle diagram for $(S^2 \times D^2) \# \CP^2$, relative to the boundary $S^2 \times S^1$. The $4$-handle is not shown. The $2$-handle corresponding to $\CP^1\subseteq \CP^2$ is shown in blue.  (b)  A handle diagram for the same manifold, but relative to a different parametrisation of the boundary, where the parametrisation has been altered by the Gluck twist. The grey arrow indicates a handle slide which produces the bottom figure. (c) The outcome of the handle slide shown in (b). The grey arrow indicates a further handle slide, which recovers the picture in (a).}
\label{fig:local-handle-diagram}
\end{figure}

Ignoring the grey arrow, \cref{fig:local-handle-diagram}\,(b) shows a handle diagram for the same manifold, but where the parametrisation of the boundary has been modified by the inverse of the Gluck twist $G^{-1}$. To see this, note that the diagram consisting of the unknot labelled with $\langle 0 \rangle$ determines an implicit parametrisation of $S^1 \times S^2$. The rotation in the definition of $G^{-1}$ changes that parametrisation in such a way that the initial $0$-framing on the left hand circle in \cref{fig:local-handle-diagram}\,(a) becomes a $+1$-framing in \cref{fig:local-handle-diagram}\,(b).  Thus the $4$-manifolds shown in (a) and (b) are diffeomorphic, via a diffeomorphism $\psi \colon (S^2 \times D^2) \# \CP^2 \xrightarrow{\cong} (S^2 \times D^2) \# \CP^2$ which, using the given parametrisations of the boundary, restricts on $S^2 \times S^1$ to $G^{-1}$.

Next we perform two handle slides, starting with the diagram in~(b). The first slide is indicated by the arrow in \cref{fig:local-handle-diagram}\,(b), and yields \cref{fig:local-handle-diagram}\,(c). The second slide is indicated by the arrow in \cref{fig:local-handle-diagram}\,(c), and produces the diagram in \cref{fig:local-handle-diagram}\,(a).
This sequence of handle slides gives rise to a rel.\ boundary diffeomorphism $\phi$ from the 4-manifold presented in \cref{fig:local-handle-diagram}\,(b) to the 4-manifold  presented in \cref{fig:local-handle-diagram}\,(a). To see how handle sliding a $2$-handle $h$ gives rise to a diffeomorphism, let $V$ denote $S^2 \times S^1 \times [0,1]$ union the 2-handles other than $h$, and let $C \cong \partial V \times [0,1]$ be a collar neighbourhood of $\partial V$. Write $h'$ for the slid 2-handle, i.e.\ the 2-handle $h$ after sliding.  Note that a handle slide of $h$ can be achieved by an ambient isotopy $\vartheta \colon \partial V \times [0,1] \to \partial V$. Define a diffeomorphism $V \cup h \to V \cup h'$ to be the identity on $V \sm C$, to be the ambient isotopy $\vartheta$ on $C$, and to be the natural extension $h \to h'$ on the 2-handle.  This extends over the $4$-handle to give the desired diffeomorphism.  The extension over the 4-handle is not necessarily uniquely determined, not even up to isotopy, since it is unknown whether the smooth mapping class group of the 4-ball rel.\ boundary is trivial. This will not be important since we will only need to control the diffeomorphism on the 2-skeleton.

Composition of the diffeomorphisms discussed in the previous two paragraphs gives rise to a diffeomorphism $\Phi := \phi \circ \psi \colon (S^2 \times D^2) \# \CP^2 \xrightarrow{\cong} (S^2 \times D^2) \# \CP^2$ that restricts on the boundary to the inverse of the Gluck twist $G^{-1} \colon S^2 \times S^1 \xrightarrow{\cong} S^2 \times S^1$.
Since the $2$-handle corresponding to the $+1$-framed unknot in \cref{fig:local-handle-diagram}\,(b) was slid over the $S^2$-factor in $S^2 \times S^1$ (corresponding to the curve labelled  $\langle 0 \rangle$), and was involved in no other handle slides, $\Phi$ sends $\CP^1$ to a sphere ambiently isotopic to $S \# \CP^1$. Now extend the diffeomorphism $\Phi$ by the identity over $E_S$ to obtain the desired homeomorphism
\[\Psi_S := \Phi \cup \Id \colon M_S \# \CP^2= ((S^2 \times D^2)\# \CP^2) \cup_G E_S \to M \# \CP^2 = ((S^2 \times D^2)\# \CP^2) \cup_{\Id} E_S. \]
This is well-defined on $S^2 \times S^1$ because $\Phi|_{S^2 \times S^1} = G^{-1}$.
Since $\Phi$ sends $\CP^1$ to a sphere ambiently isotopic to $S \# \CP^1$, after an isotopy we may assume that $\Psi_S(\CP^1) = S \# \CP^1$, as asserted.
\end{proof}

The next lemma concerns the result of blowing down a $4$-manifold along homotopic $2$-spheres, both with normal bundle having euler number $+1$, i.e.\ each sphere has a pushoff which intersects it precisely once. A similar argument applies to such spheres with normal bundle having euler number $-1$. The combination of \cref{lemma:CP2-stable-diffeo,lem:blow-down-homotopic} gives the proof of \cref{thm:gluck-she}, as we show at the end of the section.

\begin{lemma}
\label{lem:blow-down-homotopic}
Let $N$ be a compact 4-manifold.
Let $A,B\subseteq N$ be locally flat, embedded, and homotopic $2$-spheres, both with normal bundle having euler number $+1$.
Let $N^A$ and $N^B$ denote the result of blowing down $N$ along $A$ and $B$ respectively. Then $N^A$ and $N^B$ are simple homotopy equivalent.
\end{lemma}

\begin{proof}
The upcoming proof is inspired by an argument due to Stong~\cite{Stong-conn-sum}.
For an arbitrary $4$-manifold $W$ and $\beta \in \pi_2(W)$, let $\capp(W,\beta)$ denote the result $W \cup_{\beta} D^3$ of adding a 3-cell to $W$ along $\beta$.
By the definition of the blow down operation, there are canonical homeomorphisms $N^A\#\CP^2\cong N$ and  $N^B\#\CP^2 \cong N$ sending $\CP^1$ to $A$ and $B$ respectively.
Let $\sigma \in \pi_2(\CP^2)\cong \Z$ denote a generator, which we consider as an element in both $\pi_2(N^A\#\CP^2)$ and $\pi_2(N^B\#\CP^2)$, via the inclusion induced maps. Note that
\begin{equation}\label{eqn:simple-simple-hom-equiv}
  \capp(\CP^2,\sigma) \simeq_s S^4,
\end{equation}
where $\simeq_s$ denotes simple homotopy equivalence. To see this extend the standard degree one map $\CP^2 \to S^4$ over $\capp(\CP^2,\sigma)$ and observe that both spaces are simply connected and that the map induces an isomorphism on all homology groups.  Therefore $\capp(\CP^2,\sigma) \simeq S^4$ by the Whitehead theorem. Since the Whitehead group of the trivial group is trivial, we deduce that \eqref{eqn:simple-simple-hom-equiv} holds.  Now we make the following computation:
  \begin{align*}
    N^B &\cong N^B \# S^4 \simeq_s N^B \# \capp(\CP^2,\sigma) \cong \capp(N^B\#\CP^2,\sigma) \\
    &\cong \capp(N,B) \simeq_s \capp(N,A)\cong \capp(N^A\#\CP^2,\sigma) \cong N^A\# \capp(\CP^2,\sigma)\\
    &\simeq_s N^A \# S^4 \cong N^A.
    \end{align*}
 Since the computation is symmetrical in $A$ and $B$ it suffices to justify this as far as the $\capp(N,A)$ term. The first equivalence is immediate. The second equivalence uses \eqref{eqn:simple-simple-hom-equiv}. The third equivalence uses that we may perform the connected sum away from $\sigma$. The fourth equivalence uses the canonical homeomorphism $N^B\#\CP^2 \cong N$ discussed above.
    The fifth, central simple homotopy equivalence uses that $A$ and $B$ are homotopic.  The remaining equivalences can be justified by switching the r\^{o}les of $A$ and $B$, and applying the previous justifications in reverse order.
This gives the desired simple homotopy equivalence $N^A \simeq_s N^B$. The boundaries of $N$, $N^A$, and $N^B$ are naturally identified, and by inspection our simple homotopy equivalence restricts to the identity on the boundary. This completes the proof of \cref{lem:blow-down-homotopic}.
\end{proof}

\begin{proof}[Proof of \cref{thm:gluck-she}]
By \cref{lemma:CP2-stable-diffeo}, $M_S$ and $M_T$ are obtained as blow downs of $N:=M\#\CP^2$ along the homotopic spheres $A:= S\#\CP^1$ and $B:= T\#\CP^1$. By \cref{lem:blow-down-homotopic} it follows that $M_S$ and $M_T$ are  simple homotopy equivalent.
\end{proof}

\section{Homeomorphism}\label{sec:gluck-homeo}

In this section we prove \cref{cor:gluck-homeo}, which upgrades simple homotopy equivalence to homeomorphism for certain fundamental groups. We prove the two items in the corollary  separately.

\begin{proof}[Proof of \cref{cor:gluck-homeo}\,(\ref{item:i})]
Suppose that $M$ is closed and orientable, and that $\pi_1(M)$ is cyclic. Since $S$ and $T$ are homotopic, the manifolds $M_S$ and $M_T$ are simple homotopy equivalent  by \cref{thm:gluck-she}. For every cyclic group $G$,
one of Freedman~\cite{F}, Freedman--Quinn~\cite{FQ}, or
Hambleton--Kreck~\cites{Hambleton-Kreck:1988-1,Hambleton-Kreck-93}
showed that given closed $4$-manifolds $X$ and $Y$ with $X \simeq_s Y$ and $\pi_1(X) \cong \pi_1(Y) \cong G$, we have that $X$ and $Y$ are homeomorphic if and only if $\ks(X) = \ks(Y)$.
The Kirby--Siebenmann invariant $\ks$ is additive with respect to gluing along part of the boundary~\citelist{\cite{FQ}*{Section~10.2B}\cite{guide}*{Theorem~8.2\,(5)}}. This additivity is indifferent to the precise gluing homeomorphism used, so we have
\begin{equation}\label{eq:ks}
\ks(M) = \ks(E_S) + \ks(S^2 \times D^2) = \ks(M_S),
\end{equation}
and similarly $\ks(M) = \ks(M_T)$.
Therefore $M_S$ and  $M_T$ are homeomorphic, as desired.
\end{proof}

The proof of \cref{cor:gluck-homeo}\,\eqref{item:iii} will use the simple surgery exact sequence. For details and definitions of the latter, we refer the reader to~\citelist{\cite{Wall-surgery-book}\cite{FQ}*{Chapter~11}\cite{Freedman-book-surgery}}. Before giving the proof we discuss the conditions in \cref{cor:gluck-homeo}\,\eqref{item:iii}.

\begin{remark}\label{rem:conditions}~
\begin{enumerate}[(a)]
\item Observe that \cref{cor:gluck-homeo}\,\eqref{item:iii} imposes no restrictions on $w$ nor on $\partial M$.
\item The symbol $\LL$ denotes the $L$-theory spectrum of the integers, whose homotopy groups are $L_*(\Z)$, and $\LL\langle 1 \rangle$ denotes the 1-connective cover. The generalised homology $H_*(\pi;\LL\langle1\rangle^w)$ denotes the $w$-twisted $\LL\langle1\rangle$-homology \cite{bluebook}*{Appendix~A}, see also \cite{KL}*{Remark~2.1}. The groups $L_*^s(\Z\pi,w)$ are the simple $L$-groups of the ring with involution $\Z\pi$, with involution determined by $\ol{g} = w(g)g^{-1}$.
   \item\label{item:b} For $r \in \{4,5\}$, the assembly map $A_r$ factors as
   \[H_r(\pi;\LL\langle1\rangle^w)\to H_r(\pi;\LL^w) \xrightarrow{FJ}  L_r^s(\Z\pi,w).\]
   The Farrell--Jones conjecture for torsion-free groups and coefficient ring $\Z$ predicts that the map $FJ$ is an isomorphism. This conjecture holds for many families of groups, including hyperbolic groups, CAT(0) groups, 3-manifold groups, and solvable groups \cites{bartelslueckreich-FJ,bartelslueck-FJ,bartelsfarrellllueck-FJ,wegner-FJ}. In particular, by comparing the Atiyah--Hirzebruch spectral sequences for $H_r(\pi;\LL\langle1\rangle^w)\to H_r(\pi;\LL^w)$, using the naturality with respect to change in generalised homology theory, conditions \ref{item:A4} and \ref{item:A5} are satisfied for all groups of dimension at most $4$ that satisfy the Farrell--Jones conjecture~\cite{KL}*{Lemma~2.3}.
   \item  Condition~\eqref{item:iii} holds for the solvable Baumslag--Solitar groups $B(k)$, for $k\in \Z$, with any orientation character $w$; cf.\ \cite{HKT}. These groups have presentation $B(k)\cong \langle a,b\mid aba^{-1}=b^k\rangle \cong \Z[\tmfrac{1}{k}] \rtimes \Z$, and are 2-dimensional.  For these groups one can compute directly that the presentation complex for the above presentation is aspherical. In fact the presentation complex is aspherical for all torsion-free one relator groups~\cite{Lyndon}.
       Note that $B(1)\cong \Z\oplus \Z$.
    \item Condition~\eqref{item:iii} also holds for $\Z$ with both  orientation characters $w \colon \Z \to C_2$. For closed $M$ the result can be deduced from the classification theorems in \cite{FQ}*{Theorem~10.7A} and \cite{Wang-nonorientable}.
    \item\label{item:closedcase} For closed $M$, let $[M]$ denote the $\Z/2$-fundamental class and  consider the diagram
    \[\begin{tikzcd}
H^1(M;\Z/2)\ar[d,"{-\cap [M]}","\cong"']&H^1(\pi;\Z/2)\ar[l,"(c_{M})^*"',"\cong"]\ar[d,"{-\cap (c_{M})_*[M]}"]
\\
H_3(M;\Z/2)\ar[r,"(c_{M})_*"]&H_3(\pi;\Z/2)
\end{tikzcd}\]
    The bottom horizontal map is surjective if and only if
    \[-\cap (c_{M})_*[M] \colon H^1(\pi;\Z/2) \to H_3(\pi;\Z/2)\] is surjective, so we can replace condition~\ref{item:H} with the above condition. Note this means condition~\ref{item:H} does not hold for closed $M$ and $3$-dimensional $\pi$ (e.g.\ $\Z^3$), because in that case $(c_M)_*[M] \in H_4(\pi;\Z/2)=0$.
    \item Suppose that $\pi$ is a good group that satisfies the Farrell--Jones conjecture. Also suppose that $\pi$ is a $4$-dimensional Poincar\'{e} duality group with respect to $w$.
     If $M$ is closed and $(c_M)_*[M] \neq 0 \in H_4(\pi;\Z/2)$, then using the reformulation of condition~\ref{item:H} from \eqref{item:closedcase}, and using \eqref{item:b} to obtain conditions~\ref{item:A4} and~\ref{item:A5}, it follows that condition~\eqref{item:iii} is satisfied. This applies in particular when $\pi=\Z^4$, the manifold $M$ is closed, and $(c_M)_*[M] \neq 0 \in H_4(\pi;\Z/2)$.
\end{enumerate}
\end{remark}

\begin{proof}[Proof of~\cref{cor:gluck-homeo}\,(\ref{item:iii})]
Since $S$ and $T$ are homotopic, by~\cref{thm:gluck-she} the  $4$-manifolds $M_S$ and $M_T$ are simple homotopy equivalent. Recall from \cref{lem:pi1} that $\pi_1(M_S)\cong\pi_1(M)=:\pi$.
We will use the simple surgery sequence for $M_S$: \[
\begin{tikzcd}
\mathcal{N}(M_S \times [0,1],\partial)\arrow[r,"\sigma_5"]	&L_5^s(\Z\pi,w)\arrow[r]	&\mathcal{S}^s(M_S,\partial M_S)\arrow[r,"\eta"]	&\mathcal{N}(M_S,\partial M_S) \arrow[r,"\sigma_4"]	&L_4^s(\Z\pi,w),
\end{tikzcd}
\]
where the symbol $\partial$ is shorthand for $\partial:=\partial (M_S\times [0,1])=(M_S\times \{0,1\})\cup (\partial M_S\times [0,1])$
and $w \colon \pi \to C_2$ is the orientation character. Since we assume that $\pi$ is a good group (Condition~\ref{item:G}), the sequence is exact~\citelist{\cite{FQ}*{Chapter~11}\cite{Freedman-book-surgery}}. By~\cref{thm:gluck-she} there is a simple homotopy equivalence $f\colon M_T\to M_S$ restricting to a homeomorphism $\partial M_T \cong \partial M_S$. We will show that $[f]=[g\colon M_S\to M_S]$ in the simple structure set $\mathcal{S}^s(M_S,\partial M_S)$, for some self homotopy equivalence $g$ of $M_S$. By definition of the simple structure set, this means that there is an $s$-cobordism rel.\ boundary between $M_T$ and $M_S$, which implies that $M_T$ and $M_S$ are homeomorphic rel.\ boundary by the $s$-cobordism theorem~\cite{FQ}*{Theorem~7.1A}. In fact we will show that every element $[f'\colon M'\to M_S]$ of $\mathcal{S}^s(M_S,\partial M_S)$ with $\ks(M')=\ks(M_S)$ is represented by a self homotopy equivalence of $M_S$. Since $\ks(M_T) = \ks(M_S)$, as argued in~\eqref{eq:ks}, this will complete the proof that $M_S$ and $M_T$ are homeomorphic.

Using the identification \[\cN(M_S,\partial M_S)\cong [(M_S,\partial M_S),(G/TOP,\ast)]\cong H^0(M_S,\partial M_S;\LL\langle 1\rangle)\cong H_4(M_S;\LL\langle1\rangle^w),\] and the similar identification \[\cN(M_S\times[0,1],\partial)\cong H_5(M_S\times [0,1];\LL\langle 1\rangle^w)\cong H_5(M_S;\LL\langle 1\rangle^w),\] we identify the surgery obstruction maps $\sigma_r$, for $r \in \{4,5\}$, with the compositions
\begin{equation}\label{eq:surgery-obs-decompose}
H_r(M_S;\LL\langle1\rangle^w)\xrightarrow{(c_{M_S})_r} H_r(\pi;\LL\langle1\rangle^w)\xrightarrow{A_r}L_r^s(\Z\pi,w);
\end{equation}
see \cite{bluebook}*{B9,~p.~324}. By assumption (Condition~\ref{item:A5}), the map $A_5$ is surjective. We want to show that $(c_{M_S})_5$ is surjective.
Recall that
\[\pi_n(\LL\langle1\rangle)=\begin{cases} \Z & n=4k>1\\
\Z/2 & n=4k+2>1\\
0 & \text{otherwise.}\end{cases}\]
The (generalised) homology groups $H_5(M_S;\LL\langle1\rangle^w)$ and $H_5(\pi;\LL\langle1\rangle^w)$ can be computed using the Atiyah--Hirzebruch spectral sequence, and the map $c_{M_S}$ induces maps between the corresponding $E^2$ and $E^\infty$ pages. The only relevant maps between the $E^2$ pages for the computation of $(c_{M_S})_5$ from \eqref{eq:surgery-obs-decompose} are the maps $(c_{M_S})_1 \colon H_1(M_S;\Z^w)\to H_1(\pi;\Z^w)$ and $(c_{M_S})_3 \colon H_3(M_S;\Z/2)\to H_3(\pi;\Z/2)$. It will suffice to show that each is surjective. The first map is an isomorphism because $c_{M_S} \colon M_S \to B\pi$ is $2$-connected. For the second map, we use the following diagram.
\[\begin{tikzcd}
H_3(M_S\# \CP^2;\Z/2) \ar[rr,"\cong"] \ar[d,"\cong"] && H_3(M\#\CP^2;\Z/2) \ar[d,"\cong"] \\
H_3(M_S;\Z/2) \ar[r,"(c_{M_S})_3"] & H_3(\pi;\Z/2) & H_3(M;\Z/2). \ar[l,"(c_M)_3"']
\end{tikzcd}\]
The top horizontal map above is induced by the diffeomorphism $\Psi_S\colon M_S\# \CP^2 \xrightarrow{\cong} M\# \CP^2$ from \cref{lemma:CP2-stable-diffeo}. The vertical maps are induced by the standard degree one maps $M_S \# \CP^2 \to M_S$ and $M \# \CP^2 \to M$. The diagram commutes for some choice of classifying maps because both of the maps $M_S \# \CP^2 \to B\pi$, corresponding to the two routes around the diagram for the associated maps of spaces, are classifying maps for the universal cover.  It then follows from the diagram, and our assumption (Condition~\ref{item:H}) that $H_3(M;\Z/2)\xrightarrow{(c_{M})_3} H_3(\pi;\Z/2)$ is surjective,
 that $(c_{M_S})_3 \colon H_3(M_S;\Z/2)\to H_3(\pi;\Z/2)$ is also surjective, as desired.  This completes the argument that the map $(c_{M_S})_5 \colon H_5(M_S;\LL\langle1\rangle^w)\xrightarrow{} H_5(\pi;\LL\langle1\rangle^w)$ in \eqref{eq:surgery-obs-decompose} is surjective. Hence, using Condition~\ref{item:A5}, the surgery obstruction map $\sigma_5 = A_5 \circ (c_{M_S})_5$ is surjective, as needed. It follows that $\eta \colon \mathcal{S}^s(M_S,\partial  M_S) \to \mathcal{N}(M_S,\partial M_S)$ is injective.

We therefore need to investigate $\im \eta = \ker \sigma_4$.
By assumption (Condition~\ref{item:A4}), the assembly map $A_4$ is injective, so $\ker \sigma_4 = \ker (c_{M_S})_4$. We will describe a map inducing an isomorphism $\ker (c_{M_S})_4 \cong \ker (c_{M_S})_2$. To prove this, we consider the following diagram,
arising from comparing the Atiyah--Hirzebruch spectral sequences computing $H_4(M_S;\LL\langle1\rangle^w)$ and $H_4(B\pi;\LL\langle1\rangle^w)$.
\[\begin{tikzcd}
    H_3(M_S ;\Z/2) \arrow[r,"d_3"] \arrow[d,"(c_{M_S})_3",two heads] & H_0(M_S;\Z^w) \arrow[r] \arrow[d,"\cong"',"(c_{M_S})_0"] & H_4(M_S;\LL\langle1\rangle^w) \arrow[r,"h"] \arrow[d,"(c_{M_S})_4"] & H_2(M_S;\Z/2) \arrow[r] \arrow[d,"(c_{M_S})_2"] & 0 \\
    H_3(\pi ;\Z/2) \arrow[r,"d_3"] & H_0(\pi;\Z^w) \arrow[r] & H_4(\pi;\LL\langle1\rangle^w) \arrow[r]  & H_2(\pi;\Z/2) \arrow[r] & 0.
  \end{tikzcd}\]
The rows are exact and the diagram commutes.  There are two nonzero terms on the 4-line of the $E^2$ pages, corresponding to the second and fourth terms of each row. There are no differentials interacting with the $H_2(-;\Z/2)$ terms, and there is one potential $d_3$ differential, as indicated.  The map $(c_{M_S})_0$ is an isomorphism.  The $d_3$ maps are trivial when $w=0$, and when $w \neq 0$ we use that $(c_{M_S})_3$ is surjective to deduce that the images of $d_3$ are identified under $(c_{M_S})_0$.  It follows that we have a map of short exact sequences
\[\begin{tikzcd}
 0 \arrow[r] & H_0(M_S;\Z^w)/\im(d_3) \arrow[r] \arrow[d,"\cong"',"(c_{M_S})_0"] & H_4(M_S;\LL\langle1\rangle^w) \arrow[r,"h"] \arrow[d,"(c_{M_S})_4"] & H_2(M_S;\Z/2) \arrow[r] \arrow[d,"(c_{M_S})_2"] & 0 \\
0 \arrow[r] & H_0(\pi;\Z^w)/\im(d_3) \arrow[r] & H_4(\pi;\LL\langle1\rangle^w) \arrow[r]  & H_2(\pi;\Z/2) \arrow[r] & 0.
  \end{tikzcd}\]
By the snake lemma, and since the kernel of the left vertical map is trivial, the map $h$ induces an isomorphism $\ker (c_{M_S})_4 \cong \ker (c_{M_S})_2$, as asserted.    We therefore investigate the kernel of the map $(c_{M_S})_2\colon H_2(M_S;\Z/2)\to H_2(\pi;\Z/2)$.

\begin{claim}\label{claim-pre-pinch}
Every element of $\ker \big(H_2(M_S;\Z/2)\to H_2(\pi;\Z/2)\big)$ can be represented by a map of a $2$-sphere in $M_S$.
\end{claim}

\begin{proof}[Proof of \cref{claim-pre-pinch}]
By considering the Leray--Serre spectral sequence for the fibration $\wt{M}_S \to M_S \to B\pi$ with homology theory $H_*(-;\Z/2)$, where $\widetilde{M_S}$ denotes the universal cover of $M_S$, we can obtain a $\Z/2$ version of the Hopf exact sequence:
\[H_0(\pi;H_2(\wt{M}_S;\Z/2)) \to H_2(M_S;\Z/2) \to H_2(\pi;\Z/2) \to 0.\]
The first term is isomorphic to $\Z/2 \otimes_{\Z\pi} \pi_2(M_S)$. Therefore every element of $\ker \big(H_2(M_S;\Z/2)\to H_2(\pi;\Z/2)\big)$ can be represented by a sum of spheres, and so ultimately by a single sphere, in~$M_S$.
\end{proof}

By the proof of \cite{Cochran-Habegger}*{Theorem~5.1}, for every element $\iota$ of
$H_2(M_S;\Z/2)$ which is represented by an immersed sphere with $w_2(\iota)=0$, there exists a self homotopy equivalence $g$ of $M_S$ with $\eta(f)=\iota$.
Since $w_2$ is additive, the quotient of $\mathcal{S}^s(M_S,\partial M_S)$ by self homotopy equivalences has at most 2 elements.
If $w_2(\wt{M}_S)=0$, then $w_2$ is zero on all spheres, and so all elements of $\mathcal{S}^s(M_S,\partial M_S)$ are represented by self homotopy equivalences. If $w_2(\wt{M}_S)\neq 0$,
there exists a star partner $*M_S$ of $M_S$ with $\ks(*M_S)\neq \ks(M_S)$ and a simple homotopy equivalence $*M_S\to M_S$, by \cite{KPR-table}*{Proposition~5.8~and~Lemma~5.6}, cf.~\citelist{\cite{FQ}*{Chapter~10.4}\cite{Stong-conn-sum}*{Section~2}}.
Since the Kirby--Siebenmann invariant $\ks$ is a homeomorphism invariant, the maps
$\Id \colon M_S \to M_S$ and $\star M_S \to M_S$ represent two distinct elements in the quotient of $\mathcal{S}^s(M_S,\partial M_S)$ by self homotopy equivalences.
It follows that every element $[f'\colon M'\to M_S]$ of $\mathcal{S}^s(M_S,\partial M_S)$ with $\ks(M')=\ks(M_S)$ is represented by a self homotopy equivalence, as claimed.
This completes the proof that the simple homotopy equivalence $M_T \to M_S$ lies in the class of $\Id_{M_S}$ in $\mathcal{S}^s(M_S,\partial M_S)$, and therefore $M_S$ and $M_T$ are homeomorphic.
\end{proof}

\section{Fake and exotic manifolds from Gluck twists on homotopic spheres}\label{sec:fake-exotic}

In this section we prove \cref{thm:gluck-giving-fake,thm:gluck-gives-exotic}. Let $E$ denote the total space of the unique $S^2$-fibre bundle over $\RP^2$ that has orientable but not spin total space. This is a smooth, closed, orientable 4-manifold with fundamental group $\Z/2$.
\cref{thm:gluck-giving-fake} is a consequence of the following result. Recall that two $4$-manifolds $M$ and $N$ are said to be \emph{stably homeomorphic} if there exist $n,n'$ such that $M\# n(S^2\times S^2)\cong N\# n(S^2\times S^2)$.

\begin{proposition}
\label{thm:gluck-E}
There exists a pair of locally flat, embedded, and homotopic $2$-spheres $S$ and $T$ in $M:=E\# E\#\CP^2\#\ol{\CP^2}$, both with trivial normal bundle, such that $M_S$ and $M_T$ are simple homotopy equivalent but not stably homeomorphic.
\end{proposition}

For the proof we need some results on star partners of manifolds which we recall next. A \emph{star partner} of a $4$-manifold $M$ is a manifold $\star M$ such that there exists a homeomorphism $M\#\star\CP^2\cong \star M\#\CP^2$ preserving the decomposition on $\pi_2$, where $\star \CP^2$ is the Chern manifold constructed by Freedman~\cite{F}. Notably, $\star\CP^2$ is (simple) homotopy equivalent, but not homeomorphic, to $\CP^2$, and we know that $\ks(\star\CP^2)=1$. The $4$-manifolds $M$ and $\star M$ are simple homotopy equivalent but have opposite Kirby--Siebenmann invariants. The manifold $E$ has a star partner~$\star E$  \citelist{\cite{Teichner:1997-1}*{Proposition~1}\cite{Stong}}. By \cite{Teichner:1997-1}*{Theorem~1} the star partner $\star E$ of $E$ is unique up to homeomorphism.

We will also need the following observation to ensure that the constructed spheres are indeed homotopic.
\begin{remark}
\label{rem:cp2}
Note that the map $\CC^3\to\CC^3$ sending $(x,y,z)$ to $(\overline{x},\overline{y},\overline{z})$ induces an orientation preserving self-diffeomorphism of $\CP^2$ that restricts to an orientation reversing self-diffeomorphism of $\CP^1$.
\end{remark}

\begin{proof}[Proof of \cref{thm:gluck-E}]
By \cite{KPR-table}*{Proposition~5.9}, we know that $E$ is also a star partner of $\star E$, i.e.\ there is a homeomorphism $\star E\#\star\CP^2\cong E\#\CP^2$ preserving the decomposition on~$\pi_2$. Hence there exists a homeomorphism $\psi\colon \star E\#\star E\#\CP^2\to E\# E\#\CP^2$ preserving the decomposition on~$\pi_2$, obtained as the composition
\[
E\# (E \# \CP^2) \cong E\# (\star E \# \star \CP^2)= \star E\# (E \# \star \CP^2)\cong \star E\#(\star E\# \CP^2).
\]
 Let $C\subseteq E\# E\#\CP^2$ be the image of $\CP^1\subseteq \CP^2$ under $\psi$. Using the diffeomorphism from \cref{rem:cp2} to change orientations if necessary, we can assume that the locally flat embedded sphere $C$ is homotopic to $\CP^1\subseteq E\# E\#\CP^2$.
Consider the locally flat embedded spheres \[S:=\CP^1\#\ol{\CP^1} \text{ and } T:=C\#\ol{\CP^1}\] in the manifold $M := E\# E\#\CP^2\#\ol{\CP^2}$. The Gluck twist on $\CP^1\#\ol{\CP^1}$ in $\CP^2\#\ol{\CP^2}$ yields $S^2\times S^2$. It is the inverse to the Gluck twist on $S^2\times \{p\}\subseteq S^2\times S^2$  that yields $S^2\wt\times S^2\cong \CP^2\#\ol{\CP^2}$. It follows that $M_S$ is diffeomorphic to $E\# E\# (S^2\times S^2)$.

On the other hand, $M_T$ is homeomorphic to the result of the Gluck twist on $\CP^1\#\ol{\CP^1}$ in $\star E\# \star E\#\CP^2\#\ol{\CP^2}$ by construction of $T$. Hence $M_T$ is homeomorphic to $\star E\# \star E\#(S^2\times S^2)$. By \cite{Teichner:1997-1}*{Proposition~3} the manifolds $E\# E$ and $\star E\#\star E$ are not stably homeomorphic. Hence the manifolds $M_S$ and $M_T$ are not stably homeomorphic.
\end{proof}

\begin{lemma}
	\label{lem:fixing-x}
	Let $N:=\RP^4\#\CP^2$ and let $\lambda_N$ be the equivariant intersection form. Let $x,y\in \pi_2(N)$ with $\lambda_N(x,x)=\lambda_N(y,y)=1$. Then $x=\pm y$.
\end{lemma}
\begin{proof}
	We have $\pi_2(N)\cong \Z[\Z/2]$. By sesquilinearity of the intersection form, $\lambda_N(a+bT,a'+b'T)=aa'-bb'+(ba'-ab')T$, where $T\in\Z/2$ is the generator and  noting that $w(T)=-1$, for the orientation character $w\colon \pi_1(N)\to C_2$. Hence $\lambda_N(a+bT,a+bT)=a^2-b^2$. The only possible solutions in $\Z$ for $1=a^2-b^2$ are $a=\pm 1$ and $b=0$. Hence $x=\pm y$ as claimed.
\end{proof}

We are now ready for the proof of \cref{thm:gluck-gives-exotic}. Recall that $R$ denotes the Cappell--Shaneson exotic $\RP^4$, i.e.\  $R$ is homeomorphic not diffeomorphic to $\RP^4$.

\begin{proof}[Proof of \cref{thm:gluck-gives-exotic}]
Recall that $S^2 \wt{\times} S^2 \cong \CP^2\#\ol{\CP^2}$. We define $M:= \RP^4\# (S^2 \wt{\times} S^2) \cong \RP^4\#\CP^2\#\ol{\CP^2}$.
By \cite{Akbulut-fake-4-manifold}*{Theorem~2}, $R\#\CP^2$ and $\RP^4\# \CP^2$ are diffeomorphic. By \cref{rem:cp2,lem:fixing-x}, there is a diffeomorphism sending $\CP^1$ to a smooth sphere $U\subseteq \RP^4\# \CP^2$ homotopic to $\CP^1\subseteq \CP^2$. Hence as in the proof of \cref{thm:gluck-E}, for the spheres $S:=\CP^1\#\ol{\CP}^1$ and $T:=U\#\ol{\CP}^1$ in $M$ we get $M_S\cong \RP^4\#(S^2\times S^2)$ and $M_T\cong R\# (S^2\times S^2)$. Since $\RP^4$ and $R$ are homeomorphic~\cite{HKT-nonorientable} but not stably diffeomorphic~\citelist{\cite{CS-on-4-d-surgery}*{Theorem~2.4}\cite{CS-new-four-mflds-annals}*{p.~61}}, the same applies to these manifolds.
\end{proof}

\def\MR#1{}
\bibliography{bib}
\end{document}